\documentclass{article}
\usepackage[utf8]{inputenc}
\usepackage[margin=1in]{geometry}

\usepackage{algorithm}
\usepackage{algorithmic}
\usepackage{amsmath}
\usepackage{amsthm}
\usepackage{amssymb}
\usepackage{color}
\usepackage{mathtools}
\usepackage{makecell}
\usepackage{bm}
\usepackage{todonotes}
\usepackage{diagbox}
\usepackage{tablefootnote}

\DeclareMathOperator{\Rad}{Rad}

\let\baraccent=\= 
\renewcommand{\=}[1]{\stackrel{#1}{=}} 

\providecommand{\RR}{\mathbb{R}}

\providecommand{\LL}{\mathbb{L}}
\providecommand{\ZZ}{\mathbb{Z}}

\providecommand{\cA}{\mathcal{A}}
\providecommand{\cB}{\mathcal{B}}

\providecommand{\cG}{\mathcal{G}}

\providecommand{\cN}{\mathcal{N}}

\providecommand{\cY}{\mathcal{Y}}
\providecommand{\PP}{\mathbb{P}}
\providecommand{\EE}{\mathbb{E}}

\mathchardef\mhyphen="2D 

\newcommand{\BSC}{\mathsf{BSC}}
\newcommand{\AWGN}{\mathsf{AWGN}}
\newcommand{\iid}{\stackrel{\mathrm{IID}}{\sim}}

\providecommand{\sm}{\setminus}

\newcommand{\interior}[1]{%
  {\kern0pt#1}^{\mathrm{o}}%
}

\newtheorem{theorem}{Theorem}[section]
\newtheorem{lemma}[theorem]{Lemma}
\newtheorem{definition}[theorem]{Definition}
\newtheorem{example}[theorem]{Example}

\newtheorem{claim}[theorem]{Claim}
\newtheorem{proposition}[theorem]{Proposition}
\newtheorem{corollary}[theorem]{Corollary}

\title{An Information-Percolation Bound for Spin Synchronization on General Graphs}
\author{Emmanuel Abbe\thanks{This work was partly supported by NSF CAREER Award CCF-1552131, NSF Center for the Science of Information CCF-0939370.}, Enric Boix\thanks{This work was partly supported by NSF Center for the Science of Information, CCF-0939370.} \\ Princeton University}
\date{}

\begin{document}
\maketitle

\begin{abstract}
    This paper considers the problem of reconstructing $n$ independent uniform spins $X_1,\dots,X_n$ living on the vertices of an $n$-vertex graph $G$, by observing their interactions on the edges of the graph. This captures instances of models such as (i) broadcasting on trees, (ii) block models, (iii) synchronization on grids, (iv) spiked Wigner models. The paper gives an upper-bound on the mutual information between two vertices in terms of a bond percolation estimate. Namely, the information between two vertices' spins is bounded by the probability that these vertices are connected in a bond percolation model, where edges are opened with a probability that ``emulates'' the edge-information. Both the information and the open-probability are based on the Chi-squared mutual information. The main results allow us to re-derive known results for information-theoretic non-reconstruction in models (i)-(iv), with more direct or improved bounds in some cases, and to obtain new results, such as for a spiked Wigner model on grids. The main result also implies a new subadditivity property for the Chi-squared mutual information for symmetric channels and general graphs, extending the subadditivity property obtained by Evans-Kenyon-Peres-Schulman \cite{evans} for trees.    
\end{abstract}

\section{Introduction}
{\bf The model.} We consider the problem of reconstructing $n$ uniform spins $X_1,\dots,X_n \iid \Rad(1/2)$ living on the vertices of an $n$-vertex graph $G$, by observing their interactions on the edges of the graph. Formal definitions are in Section \ref{sec:modeldefinitions}. 
Depending on the choices of the graph and the interaction channel, this captures models such as (1) broadcasting on trees \cite{ks1,evans} (2) censored block models \cite{heimlicherlelargemassoulie,abbs}, (3) synchronization on grids \cite{our_grid}, (4) spiked Wigner models \cite{yash_sbm}. Here we refer to these as synchronization problems on different graph/channel models. 

To set a running example, consider the case where $G=K_n$ is the complete graph, and where the channel on each edge is a binary symmetric channel that flips the product of the spins with probability $p \in [0,1]$, i.e., for each $1 \le u<v \le n$, we observe $Y_{u,v}=X_uX_vZ_{uv}$, where $(Z_{uv})_{1 \le u<j \le n}$ are i.i.d.\ $\Rad(p)$, mutually independent of $(X_u)_{u \in [n]}$.

Note that the above model is also related to the Ising model in statistical physics; conditioned on the edge observations, the posterior distribution of the vertex spins is given by an Ising model. However, we will be interested here in the average-case behavior over the edge variables in the model, while results on Ising models (e.g., Dobrushin conditions for correlation decay \cite{dobrushin1968problem}) typically focus on worst-case behavior over the edge variables.\\

\noindent
{\bf The problem.} Depending on how ``rich'' the graph is, and how ``noisy'' the channel is, one may or may not be able to obtain a non-trivial reconstruction of the spins. We focus here on understanding when it is information-theoretically impossible to obtain a non-trivial reconstruction. For this purpose, we are interested in conditions for which the mutual information between the spins $X_u,X_v$ of two arbitrary vertices $u,v \in [n]$, given all the edge interaction variables $(Y_e)_{e \in E(G)}$, is vanishing as $n$ diverges:
\begin{align}
    I_{KL}(X_u;X_v \mid (Y_e)_{e \in E(G)}) = o(1). \label{mic}
\end{align}
For the models mentioned above, this implies in particular that there is no estimator of the spins that solves the so-called weak recovery problem, i.e., that gives an asymptotic correlation with the ground truth that is non-trivial; see \cite{abbe_fnt} for discussions on weak recovery. 

For the running example, if $p$ is bounded away from $1/2$, then for any pair of vertices, the information on their direct edge suffices to prevent \eqref{mic} to take place. If $p$ tends to $1/2$ fast enough, this may break down, but it is not enough to inspect the direct edge as the information may propagate along other paths in the graph.  \\

\noindent
{\bf Known techniques.} Different techniques have been developed to upper-bound quantities such as the mutual information of \eqref{mic}. In particular,
\begin{itemize}
    \item (i) {\bf Upgrading the graph.} This approach was developed for instance for the broadcasting on trees (BOT) problem in \cite{evans}. In the BOT model, a random variable is broadcast from the root down the edges of a tree, with each edge potentially flipping the variable, and the goal is to reconstruct the root variable from the leaf variables at infinite depth. See Section \ref{subsec:broadcastingontrees1} for formal definitions. One can view this as synchronization problem using an extra vertex that interacts noiselessly with all the leaf variables; see Section \ref{subsec:broadcastingontrees1} for the formal connection. To upper-bound the mutual information (corresponding to \eqref{mic}) from the root to the leaves in the case of binary variables and symmetric channels, \cite{evans} shows a {\it subadditivity} property of the mutual information over all paths from the root to the leaves, which implies the impossibility part (the ``difficult'' part) of the KS threshold. This subadditivity is a crucial component to establish the uniqueness of a threshold in this context, and is proved in \cite{evans} using an upgradation of the BOT ensemble on an arbitrary tree to a BOT ensemble on a ``stringy'' tree, where the branches of the tree are ``separated''. One of the open problems/directions mentioned in \cite{evans} is to extend such results to more general graphs that contain cycles, finding the right model. Part of the results in this paper can be viewed as such an attempt.
    
    \item (ii) {\bf Using an oracle to change the graph.} This approach was developed for instance for the stochastic block model in \cite{mossel-sbm}. We consider here the close variant called the censored block model. Take an Erd\H{o}s-R\'enyi random graph in the sparse regime, $G_n \sim G(n,c/n)$, and on edge of the graph, observe the product of the adjacent spins on an independent $\BSC_{\varepsilon}$ (as in the running example). This gives an instance of the CBM. It models scenarios where one observes a random measurement that gives positive or negative indication that the two incident `people' are in the same community or not. The model is closely related\footnote{In the SBM, the presence of an edge makes the two incident vertices be in different communities with probability $\varepsilon=b/(a+b)$, and each vertex has an expected number of $(a+b)/2=c$ neighbors; the difference between the SBM and the CBM is that a non-edge in the SBM carries a slight repulsion probability towards having the incident vertices in the same community, although the latter is negligible in various aspects.} to the SBM$(n,a/n,b/n)$ (with the parameters $c=(a+b)/2$ and $\varepsilon=b/(a+b)$), where each vertex in the graph is connected by an edge with probability $a/n$ if the adjacent vertices are in the same community, and $b/n$ otherwise. To show that it is not possible to reconstruct the communities in the SBM, \cite{mossel-sbm} upper-bounds \eqref{mic} with an oracle that reveals the labels of the vertices at small depth from vertex $u$. Using then the fact that the Erd\H{o}s-R\'enyi model is locally tree-like, \cite{mossel-sbm} reduces the problem to the BOT model discussed previously. The same proof technique applies to the CBM, as also obtained in \cite{lelarge2015reconstruction}. Note that this proof technique is particularly helpful in the CBM/SBM because the local neighborhood of a vertex is ``simpler,'' i.e., tree-like, allowing us to reduce the model from a loopy graph to known results for trees \cite{evans}. Such an approach may not help in the model discussed next. 

    \item (iii) {\bf Upgrading the channel.} This approach was used for instance for the synchronization problem on grids in \cite{our_grid}. Consider the case of BSC channels as in the running example, that flip the spins' product with probability $p$ on each edge, and upgrade each channel with an erasure channel that instead erases the product with probability $2p$, revealing otherwise the exact value. This erasure model is clearly an upgradation of the BSC model, since one can always draw a random spin in replacement to an erasure symbol, which gives a BSC of flip probability $2p/2=p$. As further discussed below, for an erasure model, the mutual information in \eqref{mic} becomes exactly the probability that $u$ and $v$ are connected in a bond percolation model. In graph models like the grid, this has either a sharp threshold or some known bounds \cite{Grimmett}, and the overall approach gives a bound for synchronization problems on grids, developed in \cite{our_grid} beyond the case of BSCs. Note however that this approach is unlikely to give a sharp bound, due to the upgradation, but it allows for a direct application of percolation bounds.  
    
    \item (iv) {\bf Interpolation, message-passing and second-moments.} Interpolation techniques take different forms; one consists in establishing a bound between two quantities by parametrizing each quantity with a relevant parameter, typically a notion of signal-to-noise ratio (SNR), establishing the bound for the boundary cases, and interpolating other cases with a ``monotonicity'' argument (inspecting a derivative). This approach has long been used in  different contexts; for example, to establish the ``entropy power inequality'' in \cite{stam}. More closely related to us, it is used in \cite{abbetoc} to establish a subadditivity property of the mutual information of graphical channels, where the subadditivity acts on the vertex-set rather than the edge-set as considered here. For the spiked Wigner model with Radamacher inputs, which corresponds to a complete graph with a Gaussian noise channel, one can use the I-MMSE formula from \cite{i-mmse} to equate the derivative of the mutual information in \eqref{mic} to the MMSE, and express the latter using an approximate message passing (AMP) estimate \cite{amp}. This allows \cite{yash_sbm} to establish a limiting expression for the mutual information, and in particular, a tight condition for when the latter vanishes. Similar techniques have been used in various other spiked Wigner models, such as in \cite{wein1,wein_tensor,banks3}, and block models \cite{coja2}. It is worth noting that if the goal is to only obtain a condition for when the mutual information vanishes, it may not be necessary to employ such elaborate estimates. In particular, one may rely on second-moment estimates as used in \cite{banks,banks2} for block models and \cite{banks3,wein1} for Spiked Wigner models. Second-moment estimates typically give conditions on when the distribution of the planted ensemble (where the edge variables depend on the $X_i$ variables) is contiguous to the unplanted ensemble (where the edge variables are independent), and depending on the model, this can be turned into a condition for weak recovery being not solvable, such as in \cite{banks,banks2,banks3,wein1} (although the implication may not be true in general). 
\end{itemize}

{\bf This paper.} As apparent in previous discussion, some of the known techniques are fairly graph- and channel-dependent. The goal of this paper is to introduce a general method to upper-bound the mutual information \eqref{mic} in terms of bond percolation estimates, namely, in terms of the probability that vertices $u$ and $v$ are connected by an ``open'' path in a model where each edge of $G$ is kept open with some probability.  

Note that if the channel on each edge is an erasure channel, i.e., if $Y_{uv}=X_uX_v$ with probability $q$ and $Y_{uv}=\star$ (an erasure symbol) with probability $1-q$, then 
\begin{align}
    I_{KL}(X_u;X_v \mid (Y_e)_{e \in E(G)}) = \PP(u \sim v \mid (q)_{e \in E(G)}), \label{erasure}
\end{align}
where $\PP(u \sim v \mid (q)_{e \in E(G)})$ denotes the probability that $u$ and $v$ are connected in a bond percolation model on $G$ where each edge is open independently with probability $q$.

Our main result shows how to turn previous equality into an inequality beyond the case of erasures, covering a fairly general family of channels that contains models (1)-(4). The crucial part is to find how to set the openness probabilities on each edge in order to ``emulate'' the right amount of information, rather than using a degradation argument as discussed in (iii) above, that produces loose bounds on models (1)-(4). For this purpose, we will use an interpolation technique. In a sense, our bound can thus be viewed as an hybrid between the techniques of \cite{abbetoc} and \cite{our_grid}, as it uses an interpolation technique for a percolation bound. 

The main feature of the bound is that it applies to {\it any} graph. The derived bound subsumes the known results for (1)-(4) (with slight improvements for (3)) and gives also a few new results. These are presented in Section \ref{sec:applications}. Discussions on how the bound could be extended beyond the binary setting are provided in Section \ref{sec:futuredirections}. We underline here two aspects of the main results: 
\begin{itemize}
    \item {\bf A Chi-squared bound.} A natural attempt to estimate the information between two vertices in terms of the probability that these vertices are connected in a bond percolation model, is to open each edge with a probability that ``emulates'' the information of the edge. How should this be formalized?   
    
    Consider the case of $G=\Pi_n$, a path on the vertices $1,2,\dots,n$, with a binary symmetric channel (BSC) of flip probability $p=(1-\delta)/2$ on each edge as in the running example. The channel between the first and last vertex ($1$ and $n$) is a concatenation of BSCs, each with a flip probability either $p$ or $1-p$ (depending of the value of $Y_{i,i+1}$ for edge $(i,i+1)$). Thus we can explicitly compute the LHS of \eqref{mic}: 
\begin{align}
    I_{KL}(X_1;X_n \mid (Y_e)_{e \in E(\Pi_n)}) = 1- H((1-\delta^{n-1})/2). 
\end{align}
On the other hand, if we open each edge in the path with probability equal to the mutual information of a $\BSC_p$ (or $\BSC_{(1-p)}$), i.e., with $q=1-H((1-\delta)/2)$, vertex $u$ and $v$ are connected with probability 
\begin{align}
\PP(u \sim v \mid (q)_{e \in E(\Pi_n)}) = (1-H((1-\delta)/2))^{n-1}.
\end{align}
Unfortunately, this gives a bound that is in the reverse direction of \eqref{mic}! Note also that one can not hope for a general bound in this reverse direction for the mutual information (e.g., one can get a counter-example on a triangle-graph).  

In order to obtain a bound that holds for arbitrary finite graphs, we will change our measure of information, using not the KL-divergence but the Chi-squared divergence, i.e., 
\begin{align}
I_2(X;Y) :=D_{\chi^2}(p_{X,Y}\| p_X p_Y) 
\end{align}
where $D_{\chi^2}$ is the Chi-squared $f$-divergence with $f(t)=(t-1)^2$. In particular, it is easily shown that that for the path example, 
\begin{align}
    I_2(X_1;X_n|(Y_e)_{e \in E(\Pi_n)}) = \delta^{2(n-1)}. \label{chi-path}
\end{align}
Therefore, opening edge with probability equal to the Chi-squared mutual information of a $\BSC_{((1-\delta)/2)}$, i.e., $\delta^2$, gives the desired upper-bound with equality.

In general, we obtain that for any graph $G$ and for a class of symmetric channels on the edges,  
\begin{align}
    I_2(X_u;X_v \mid (Y_e)_{e \in E(G)}) \le \PP( u \sim v \mid (I_2(X_e;Y_e))_{e \in E(G)} ) \quad \text{{\bf (main result) }}, \label{main1}
\end{align}
where the RHS is the probability that $u$ and $v$ are connected in a bond percolation model on $G$ where edge $e$ is open with probability $I_2(X_e;Y_e)$, where $X_e$ denotes the product, $X_i \cdot X_j$, of the spins incident to edge $e = (i,j)$. 

Further, one can go back to obtain an upper-bound for the LHS in terms of the classical mutual information, since the latter is upper-bounded by the Chi-squared mutual information (for uniform binary variables); it is however important to keep the Chi-squared mutual information on the RHS. (See Lemma \ref{lem:evanschi2vsklonradhalf}.)

\item {\bf Subadditivity for general graphs.} Note that the RHS of \eqref{main1} can be upper-bounded with the union bound over all paths between $u$ and $v$, and using \eqref{chi-path}, we obtain as a corollary the following subadditivity property for general graphs:
\begin{align}
    I_2(X_u;X_v\mid (Y_e)_{e \in E(G)}) \le \sum_{\gamma \in \Gamma_G(u,v)} I_2(X_u;X_v\mid (Y_e)_{e \in E(\gamma)}), \label{main2}
\end{align}
where $\Gamma_G(u,v)$ denotes the set of paths (i.e., self-avoiding walks) from $u$ to $v$ in $G$. This gives an extension via the synchronization model of the subadditivity obtained for trees in \cite{evans} (see point (i) above) to general graphs.  
\end{itemize}

\section{Model}\label{sec:modeldefinitions}
We begin by defining a ``graphical channel'' similarly to the definition in \cite{abbetoc}, but tailored to the binary case:
\begin{itemize}
\item Let $g = (V,E(g))$ be a finite graph with vertex set $V = [n]$ and edge set $E(g)$.

\item For each $e \in E(g)$, let $Q_e(\cdot \mid \cdot)$ be a probability transition function (channel) from the binary input alphabet $\{-1,+1\}$ to an output alphabet $\cY_e$, such that $Q_{e\mid +}(\cdot) \equiv Q_e(\cdot \mid +1)$ and $Q_{e \mid -}(\cdot) \equiv Q_e(\cdot \mid -1)$ are probability measures on a measurable space $(\cY_e,\cA_e)$.

\item Assign a vertex label $x_i \in \{-1,+1\}$ to each vertex $i \in V$. Assign an edge label $y_e \in \cY_{e}$ to each edge $e \in E(g)$. Then define the channel $P_{g,Q}(\cdot \mid \cdot)$ with input alphabet $\{-1,+1\}^{V}$ and output alphabet $\{-1,+1\}^{E(g)}$ as follows: for each measurable set $A = \prod_{e \in E(g)} A_e \in \prod_{e \in E(g)} \cA_e$, let 
$$P_{g,Q}(A \mid x) \equiv \prod_{e \in E(g)} Q_e(A_e \mid x_e),$$ where we use the notation $x_e = x_u \cdot x_v$ for $e=(u,v)$.

\end{itemize}
\begin{definition}[Graphical channel for fixed graph]
Let $g,Q$ and $P_{g,Q}$ be as above. We call $P_{g,Q}$ a graphical channel with graph $g$ and channels $Q$.
\end{definition}

\begin{definition}[Graphical channel for random graph] Let $G = (V, E(G))$ be a random graph with vertex set $V = [n]$, and let $Q$ be a collection of edge channels (as above) so that for any edge $e$, $Q_e$ is defined if $\PP(e \in E(G)) > 0$. Let $P_{G,Q}$ be the random channel with output alphabet $\prod_{e \in E(G)} \cY_e$ and input alphabet $\{-1,+1\}^{V}$ given by $P_{g,Q}$ for each realization $G = g$.

\end{definition}

\begin{definition}[Binary synchronization instance] Let $P_{G,Q}$ be an $n$-node graphical channel, and let $X$ be uniformly drawn in $\{-1,+1\}^n$. Let $Y$ be the output of $X$ through the graphical channel $P_{G,Q}$. The pair $(X,Y)$ is an instance of a binary synchronization problem drawn from $P_{G,Q}$.
\end{definition}

\section{Main Results}
In this paper, we provide progress towards answering the following question: given a binary synchronization instance $(X,Y)$ drawn from $P_{G,Q}$, for $u,v \in V$, if we know $X_v$ and we know $Y$, then when is it impossible to reconstruct $X_u$?

\subsection{$\chi^2$-mutual information} In particular, we provide an upper-bound on the information that $X_v$ and $Y$ give about $X_u$. This information is quantified by the $\chi^2$-mutual information $$I_2(X_u; X_v, Y),$$ which is the $f$-mutual information based on the $\chi^2$-divergence --- see Section \ref{app:chi2mutualinformation} for a reminder on the definitions and properties of these functionals.  

\begin{proposition}\label{prop:chi2informationequivalences}
If $(X,Y)$ is a binary synchronization instance with underlying graph $G$, and $u,v \in V(G)$, then following equality holds:
\begin{equation}I_2(X_u; X_S, Y) = I_2(X_u; X_v \mid Y).\end{equation}
\end{proposition}
The $\chi^2$-mutual information takes the following simple expression:
\begin{proposition}\label{prop:chi2binarysynchexpectationexpression}
If $(X,Y)$ is a binary synchronization instance with underlying graph $G$, and $u,v \in V(G)$, then the following equality holds:
$$I_2(X_u; X_v \mid Y) = \EE_Y[\EE_X[X_u \cdot X_v \mid Y]^2].$$
\end{proposition}

The definition of $I_2$, and the proofs of Propositions \ref{prop:chi2informationequivalences} and \ref{prop:chi2binarysynchexpectationexpression}, can be found in Appendix \ref{app:chi2mutualinformation}.

\subsection{Bond percolation}
In our main result, we bound the Chi-squared mutual information $I_2(X_u; X_v \mid Y)$ by the connection probability between $u$ and $v$ in a bond percolation on the underlying graph, $G$.

\begin{definition}[Bond percolation on a graph] Let $G = (V,E(G))$ be a graph, and let $\gamma : E(G) \to [0,1]$. Then, a bond percolation with open probability $\gamma$ on $G$ is a random edge-labelling $$B : E(G) \to \{\mbox{open}, \mbox{closed}\},$$ such that each edge label $B(e)$ is assigned independently of the other edge labels, and such that for all $e$, $$\PP[B(e) = \mbox{open} \mid e \in E(G)] = \gamma_e.$$
\end{definition}

Let $B$ be a bond percolation on $G$. If a subgraph $H \subseteq G$ is such that $B(e) = \mbox{open}$ for all $e \in E(H)$, then we call $H$ an open subgraph.
\begin{definition}[Connection probability in percolation]
Let $S,T \subseteq V(G)$. Then we write their connection probability in a percolation on $G$ with open probability $\gamma$ as
$$\PP(S \sim T \mid \gamma).$$ This denotes the probability that there is a pair of vertices $u \in S$, $v \in T$, such that $u$ is connected to $v$ by an open path in a bond percolation on $G$ with open probability $\gamma$.
\end{definition}


\subsection{Symmetric channels}
Our information-theoretic bound for spin synchronization applies to ``symmetric'' graphical channels defined as follows.

\begin{definition}\label{def:symmetricgraphicalchannel}
A graphical channel $P_{G,Q}$ is symmetric if for each edge $e \in E(G)$ the channel $Q_e(\cdot \mid \cdot)$ is symmetric. An edge  channel $Q_e(\cdot \mid \cdot)$ is symmetric if there is a measurable transformation $T_e : \cY_e \to \cY_e$ on the output alphabet of $Q_e(\cdot \mid \cdot)$ such that $T_e = T_e^{-1}$, and such that for all measurable $A \subset \cA_e$ we have $$Q_e(A\mid+1) = Q_e(T_e(A)\mid-1),$$ and hence $$Q_e(T_e(A) \mid +1) = Q_e(A \mid -1).$$
\end{definition}

In other words, an edge channel $Q_e(\cdot \mid \cdot)$ is symmetric if ``flipping the sign'' using $T_e$ of an edge label $Y_e$ with distribution $Q_{e\mid +}$ gives an edge label $T_e(Y_e)$ with distribution $Q_{e\mid -}$.

Symmetric graphical channels cover a broad collection of models, discussed in Section \ref{sec:applications}.

\subsection{Information-percolation bound}

\begin{theorem}\label{thm:mainsymmetricbound}
Let $P_{G,Q}$ be a symmetric graphical channel, where $G = (V,E(G))$ is a random graph with vertex set $V = [n]$. Let $(X,Y)$ be a binary synchronization instance drawn from $P_{G,Q}$.

Then for all $u,v \in V$, $$I_2(X_u; X_v\mid Y) \leq \PP(u\sim v \mid \gamma),$$ where $$\gamma_{(i,j)} = I_2(X_i; X_j\mid Y_{(i,j)})$$ for all $(i,j) \in E(G)$.
\end{theorem}

\begin{corollary}\label{cor:symmetricbound}
Let $P_{G,Q}$, $(X,Y)$, and $\gamma$ be as in Theorem \ref{thm:mainsymmetricbound}.

Then for all $u \in V$, $S \subseteq V$, $$I_2(X_u; X_S\mid Y) \leq \PP(u\sim S \mid \gamma).$$
\end{corollary}
We refer to Section \ref{sec:futuredirections} for discussions on how this result may be extended to more general graphical channels, in particular for more general edge channels. 

\section{Applications}\label{sec:applications}

Many common edge channels enjoy the symmetry property of Definition \ref{def:symmetricgraphicalchannel}. We discuss here some important examples. 

\paragraph{Binary Symmetric Channel} One example is the binary symmetric channel with flip probability $\varepsilon$ ($\BSC_{\varepsilon}$, for short). This channel has input and output alphabet $\{-1,+1\}$, and is given by \begin{equation*}\label{eq:defbsc}\BSC_{\varepsilon}(y \mid x) = \begin{cases} 1-\varepsilon, & x = y \\ \varepsilon, & x \neq y \end{cases}.\end{equation*} This channel is symmetric in the sense of Definition \ref{def:symmetricgraphicalchannel}, because the transformation $T(y) = -y$ satisfies both $T^2 = 1$ and $\BSC_{\varepsilon}(T(y) \mid x) = \BSC_{\varepsilon}(y \mid -x)$.

\paragraph{Additive White Gaussian Noise Channel} Another example is the Gaussian noise channel $\AWGN_{\lambda}$, whose output distribution $\AWGN_{\lambda}(\cdot \mid x)$ is the distribution of the random variable $$Y_x = \sqrt{\lambda} x + Z,$$ where $Z \sim \cN(0,1)$ is independent Gaussian noise with mean $0$ and variance $1$. This channel is also symmetric in the sense of Definition \ref{def:symmetricgraphicalchannel}, because the transformation $T(y) = -y$ satisfies $T^2 = 1$ and $\AWGN_{\lambda}(T(\cdot) \mid x) = \AWGN_{\lambda}(\cdot \mid x)$, since $-Y_x = -\sqrt{\lambda}x - Z \sim \sqrt{\lambda} (-x) + Z = Y_{-x}$, because $Z \sim -Z$.

Table \ref{tab:lowerboundstable} gives examples of information-theoretic thresholds that can be obtained as a direct consequence of Theorem \ref{thm:mainsymmetricbound}. In all of these cases, our bounds either matches or improves the previously-known bounds.\footnote{Note that \cite{our_grid} does not attempt to obtain the tightest bound, but rather the existence of a positive lower-bound on the threshold.} The table also gives a few new results.

\renewcommand{\arraystretch}{2}
\begin{table}
\small
    \centering
    \begin{tabular}{|c|c|c|c|c|}
        \hline
        \textbf{Graph\textbackslash Edge Channel}
        & \multicolumn{2}{c|}{$\BSC_{\varepsilon}$} & \multicolumn{2}{c|}{$\AWGN_{\lambda}$} \\
        \hline
        & Known Bound & Our Bound & Known Bound & Our Bound\footnotemark \\
        \hline
         Tree $T$ & \makecell{$(1-2\varepsilon)^2 \leq p_c(T)$ \\ Broadcasting on Trees \\ \cite{evans}} & \makecell{$(1-2\varepsilon)^2 \leq p_c(T)$ \\ Section \ref{subsec:broadcastingontrees1}} & & $f(\lambda) \leq p_c(T)$ \\
         \hline
         Erd\H{o}s-R\'enyi$(n,c/n)$ & \makecell{$(1-2\varepsilon)^2 \leq 1/c$ \\ Censored Block Model \\ \cite{mossel-sbm,lelarge2015reconstruction} } & \makecell{$(1-2\varepsilon)^2 \leq 1/c$ \\ Section \ref{subsec:CBMandclustering}} & & $f(\lambda) \leq 1/c$ \\
         \hline
         Grid $\LL^2$ & \makecell{$(1-2\varepsilon)^2 \leq 1/4$ \\ Grid Synchronization \\ \cite{our_grid}} & \makecell{$(1-2\varepsilon)^2 \leq 1/2$ \\ Section \ref{subsec:gridapplication}} & & $f(\lambda) \leq 1/2$ \\ 
         \hline
         Complete $K_n$ &  & $(1-2\varepsilon)^2 < 1/n$ & \makecell{$\lambda \leq c/n$ for $c < 1$ \\ Spiked Wigner \\ \cite{yash_sbm}} & \makecell{$\lambda \leq c/n$ for $c < 1$ \\ Section \ref{subsec:spikedwignerapplication}} \\
         \hline
    \end{tabular}
    \caption{Regimes in which weak recovery/reconstruction is impossible}
    \label{tab:lowerboundstable}
\end{table}
\footnotetext{Where $f(\lambda) = I_2(X_1; X_2 \mid Y^{(\lambda)})$ for $Y^{(\lambda)} = \sqrt{\lambda}X_1X_2 + Z$, and $X_1,X_2 \stackrel{i.i.d.}{\sim} \Rad(1/2)$, $Z \sim \cN(0,1)$. As calculated in \cite{abbetoc}, $f(\lambda) = \EE[\tanh(\lambda + \sqrt{\lambda}Z)^2]$.}

\subsection{Broadcasting on Trees}\label{subsec:broadcastingontrees1}
In the ``broadcasting on trees'' problem, each vertex $v \in V(T)$ of an infinite tree $T$ has a binary hidden label $\sigma_v$. The hidden labels are assigned by letting the root $\rho$ have spin $\sigma_{\rho} \sim \Rad(1/2)$, and by defining edge labels $\{\eta_e\} \stackrel{i.i.d.}{\sim} \Rad(\varepsilon)$, and letting $$\sigma_{v} = \sigma_{\rho} \prod_{e} \eta_e,$$ where the product is over the edges in the path from $\rho$ to $v$.

In \cite{evans}, it is proved that for $(1-2\varepsilon)^2 < p_c(T)$, $$I_{\mathrm{KL}}(\sigma_{\rho}; (\sigma_{v})_{\{v \ :\  d(\rho,v) = t\}}) \to 0,\quad \text{as }t \to\infty,$$ where $d(v,w)$ denotes for the distance between two vertices $v$ and $w$ in $T$ and $p_c(T)$ denotes the bond percolation threshold of $T$. In other words, for $\varepsilon$ too close to $\frac{1}{2}$, the information given by the depth-$n$ vertex labels about the root goes to 0, and hence reconstruction of the root label from the leaf labels becomes impossible. In fact, \cite{evans} showed this bound on the mutual information is tight: reconstruction is possible for $(1-2\varepsilon)^2 > p_c(T)$, which already known from \cite{ks1} in some cases. But we will only concern ourselves with the impossibility result of the paper.

\begin{example}\label{example:broadcastingontreessymmetric}
We rederive the impossibility result of \cite{evans} by applying Corollary \ref{cor:symmetricbound}.
\end{example}
\begin{proof}

The proof follows by constructing a group synchronization problem that is equivalent to the broadcasting problem.

Let $\{X_v\}_{v \in V(T) \sm \rho} \stackrel{i.i.d}{\sim} \Rad(1/2)$. Let $X_{\rho} = \sigma_{\rho}$. For each $e = (i,j) \in E(T)$ define $$Y_{ij} = X_i \cdot X_j \prod \eta_e.$$ Then $(X,Y)$ is a binary synchronization instance drawn from $P_{T,Q}$, where $Q_e$ is $\BSC_{\varepsilon}$ for each edge $e \in E(T)$. Notice that \begin{align}
    I_{\mathrm{KL}}(\sigma_{\rho}; (\sigma_v)_{\{v \ :\  d(\rho,v) = t\}}) &\leq I_2(\sigma_{\rho}; (\sigma_v)_{\{v \ :\  d(\rho,v) = t\}}) \label{eq:refchi2vskllemma} \\ &\leq I_2(X_{\rho}; (X_v)_{\{v \ :\  d(\rho,v) = t\}}, Y) \label{eq:botdataprocessing} \\ \label{ineq:invokecorollarybondpercbroadcasting} &\leq \PP[\mbox{there is length-}t \mbox{ path from } \rho \mbox{ in $(1-2\varepsilon)^2$-prob. bond perc. on } T],
\end{align}
where \eqref{eq:refchi2vskllemma} follows by Lemma \ref{lem:evanschi2vsklonradhalf}, \eqref{eq:botdataprocessing} follows by the data-processing inequality, and \eqref{ineq:invokecorollarybondpercbroadcasting} follows by Theorem \ref{cor:symmetricbound}. The bond percolation has open probability $(1-2\varepsilon)^2$. For $(1-2\varepsilon)^2 \leq p_c(T)$, the probability of a length-$t$ path vanishes as $t \to \infty$, proving the theorem.
\end{proof}

\subsection{Clustering in the Censored Block Model}\label{subsec:CBMandclustering}
Another application arises in the domain of graph clustering and community detection. Our bound applies to the Censored Block Model (CBM). 
This model is defined in \cite{abbs} for general graphs $G$ when  
the edge channel consists of BSCs, i.e., $$Y_{ij} = X_i\cdot X_j \cdot Z_{ij},$$ for each $(i,j) \in E(G)$, where $Z_{ij} \sim \Rad(\varepsilon)$ is independent noise. In the language of our paper, $(X,Y)$ is a binary synchronization instance on $G$, and all the edge channels are $\BSC_{\varepsilon}$.

\begin{example}\label{example:cbm}
Suppose $G$ is distributed as an Erd\"os-R\'enyi random graph $G(n,\frac{c}{n})$. Weak recovery is impossible in a censored block model on $G$ with flip probability $\varepsilon$ if $$c \leq 1/(1-2\varepsilon)^2.$$
\end{example}
\begin{proof}

For all $u,v \in V(G)$, by Theorem \ref{thm:mainsymmetricbound},
$$I_2(X_u; X_v\mid Y) \leq \PP(u\sim v \mid \{(1-2\varepsilon)^2\}_{e \in E(G)}) = \PP(u\sim v \mid (c(1-2\varepsilon)^2/n)_{e \in K_n}) \to 0$$ if $c \leq 1/(1-2\varepsilon)^2$, since the largest component of $G(n,c/n)$ is of size $O(n^{2/3}) = o(n)$ in this regime (by \cite{ER-seminal2}).
\end{proof}
This rederives a threshold conjectured in \cite{heimlicherlelargemassoulie} and proved in \cite{lelarge2015reconstruction}. 
The proof is analog to the proof of \cite{mossel-sbm} that establishes non-reconstruction for the two-community symmetric Stochastic Block Model $SBM(n,a/n,b/n)$ when $(a-b)^2 \le 2(a+b)$.
While \cite{lelarge2015reconstruction} does not establish the impossibility of reconstruction at the critical threshold, it is straightforward to extend the argument at the threshold. Note also that this gives a tight threshold, i.e., it is proved that reconstruction (a.k.a.\ weak recovery) is possible above this threshold \cite{new-vu,florent_CBM}.

\subsection{Grid Synchronization}\label{subsec:gridapplication}

The proof of \cite{lelarge2015reconstruction} which implies impossibility of reconstruction in the censored block model on the Erd\"os-R\'enyi random graphs $G(n,c/n)$ relies crucially on the fact that for constant $c$, most small neighborhoods of vertices in $G(n,c/n)$ are trees. 

However, the method of coupling with trees would no longer apply if we were to work with the Censored Block Model on a grid, because grids have many small cycles. In this case, our bound still goes through, and is in fact stronger than the previously-known bound of \cite{our_grid} for binary synchronization. Supposing the edge channels were binary symmetric channels with flip probability $\varepsilon$, The previous bound required $(1-2\varepsilon)^2 \leq \frac{1}{4}$ for impossibility of synchronization, while ours only requires $(1-2\varepsilon)^2 \leq \frac{1}{2}$:

\begin{example}\label{example:gridsynchronization}
Let $\LL^2$ be the two-dimensional lattice with vertices $V(\LL^2) = \ZZ^2$ and edges given by the Hamming distance. Let $v_1,\ldots,v_k,\ldots$ be a sequence of vertices such that $v_k$ is at distance $k$ from $0$. Let $(X,Y)$ be a binary synchronization instance drawn from $P_{\LL^2,Q}$, where all the edge channels are $\BSC_{\varepsilon}$. Then, if $(1-2\varepsilon)^2 \leq \frac{1}{2},$ we have $I_2(X_0; X_{v_k}\mid Y) \to 0$ as $k \to \infty$.
\end{example}
\begin{proof}
By Theorem \ref{thm:mainsymmetricbound}

\begin{align}
    I_2(X_0; X_{v_k}\mid Y) &\leq \PP(0 \sim v_k \mid ((1-2\varepsilon)^2)_{e \in E(\LL^2)}) \label{eq:applytheoremoninfinitegraphgridapplication}
    \\ &\to 0 \mbox{ as } k \to \infty \label{eq:citegrimmettforthisone},
\end{align}


Line \eqref{eq:citegrimmettforthisone} follows because $(1-2\varepsilon)^2 \leq 1/2$, which is the critical bond percolation constant of $\LL^2$. And it is known that the probability that there is an open length-$k$ path containing the origin in a critical or sub-critical bond percolation on $\LL^2$ vanishes as $k \to \infty$. A reference for this is \cite{Grimmett}.

Notice that in \eqref{eq:applytheoremoninfinitegraphgridapplication} we have applied Theorem \ref{thm:mainsymmetricbound} in the case of an infinite graph, although we have technically proved the theorem only for finite graphs. We may do this by the monotone convergence of the information and of the connection probability in the percolation.
\end{proof}

\subsection{Spiked Gaussian Wigner Model}\label{subsec:spikedwignerapplication}
In the spiked Wigner model with $\Rad(1/2)$ priors, 
we are given an $n \times n$ matrix $$Y_{\lambda} = \sqrt{\frac{\lambda}{n}} X X^T + W,$$ where $X$ is uniform in $\{-1,+1\}^n$, and $W$ is an independent Gaussian Wigner matrix (real, symmetric, the entries are distributed as unit Gaussians $\cN(0,1)$ and are all independent except for the symmetry constraint).

The spiked Wigner model, and spiked matrix models in general, have been studied in various contexts: for example, in order to evaluate statistical methods such as PCA that estimate low-rank information from noisy data, or as variants of the stochastic block model (\cite{javanmard2016phase}, \cite{wein1}, \cite{alaoui2018estimation}). For $Y_{\lambda}$ as above, \cite{yash_sbm} proved that there is a phase transition in the problem of weak recovery at exactly the critical threshold $\lambda_c = 1$. The impossibility part of this phase transition was later rederived in a more general setting by \cite{wein1}.

The impossibility of recovery for $\lambda < 1$ is a direct consequence of Theorem \ref{thm:mainsymmetricbound}:

\begin{example}\label{example:spikedwigner} Let $Y_{\lambda}$ be defined as above. Then, for $\lambda < 1$, $I_2(X_u; X_v \mid Y_{\lambda}) \to 0$ for all $u\neq v,$ and hence it is impossible to weakly recover $X$ from $Y_{\lambda}$.
\end{example}
\begin{proof}
$(X,Y_{\lambda})$ is distributed as a binary synchronization instance drawn from a graphical channel on $K_n$, in which each edge channel $Q_{(i,j)}$ is given by $$Y_{\lambda,ij} = \sqrt{\frac{\lambda}{n}}X_i\cdot X_j + Z_{ij},$$ where $Z \stackrel{i.i.d}{\sim} \cN(0,1).$
Notice that the edge channels are symmetric (with the transformation $y \mapsto -y$).
Analogously to the case of the censored block model on $G(n,c/n)$, it suffices to show that $$I_2(X_i; X_j \mid Y_{\lambda,ij}  ) = \frac{\lambda}{n} + o(1/n).$$

This is done by explicit calculation. Writing $a = \sqrt{\lambda/n}$,
\begin{align}
I_2(X_i; X_j \mid Y_{\lambda, ij}) &= \EE[\EE[X_i \cdot X_j \mid Y_{\lambda, ij}]^2] \\ &= \int_{-\infty}^{+\infty} \frac{e^{-(x-a)^2/2} + e^{-(x+a)^2/2}}{2 \sqrt{2\pi}} \left(\frac{e^{-(x-a)^2/2} - e^{-(x+a)^2/2}}{e^{-(x-a)^2/2} + e^{-(x+a)^2/2}}\right)^2 dx
\\ &\leq \int_{-\infty}^{+\infty} \frac{e^{-(x-a)^2/2} + e^{-(x+a)^2/2}}{2\sqrt{2\pi}} \left(ax\right)^2 dx \\ \label{eq:gaussianintegralforspikedwigner} &=  a^2(a^2 + 1) \\ &= \frac{\lambda}{n} + o(1/n).
\end{align}
Line \eqref{eq:gaussianintegralforspikedwigner} is a standard Gaussian integral.
\end{proof}

\section{Additional results and future directions}\label{sec:futuredirections}
As mentioned in the introduction, fixing the edge observations and applying the Ising model correlation decay conditions yields an impossibility result for reconstruction. However, the bounds that we achieve with this method are not as strong as those we proved in this paper, because the techniques in our paper allow us to deal with the average-case edge observations, while fixing the edge observations and applying the Dobrushin conditions requires us to work with the worst-case edge observations. It would nonetheless be interesting to elaborate on this connection. 

Various natural extensions can be considered for the main result of this paper. The first one concerns more general edge channels, such as non-binary input alphabets and non-symmetrical channels. We provide below a more general condition on the edge channel that would suffice for the current proof technique to work, without giving explicit examples. In the theorem below, the vertex labels are uniformly random members of some finite group $\cG$ (not necessarily $\{+1,-1\}$), and the edge labels, $Y_{(i,j)}$, are noisy observations of the differences of the endpoints, $X_i \cdot X_j^{-1}$. The proof of Theorem \ref{thm:mainabstract} is analogous to the proof of Theorem \ref{thm:mainsymmetricbound}.

\begin{theorem}\label{thm:mainabstract} Let $G = (V, E)$ be a finite and (for simplicity) deterministic graph with vertex set $V$ and edge set $E$. For every $\gamma \in [0,1]$, let $Q^{\gamma}$ be a collection of edge channels for $G$, with input alphabet $\cG$. For any $\Gamma \in [0,1]^{E}$, let $Q^{\Gamma}$ be the collection of edge channels $(Q_e^{\Gamma(e)})_{e \in E}$, and let $(X^{\Gamma}, Y^{\Gamma})$ be a group-$\cG$ synchronization instance drawn from $P_{G,Q^{\Gamma}}$.

\begin{enumerate}
\item Suppose that $$I_2(X_e^0; Y_e^0) = 0$$ for all $e \in E$.

\item Suppose also that for every $e \in E$, $u,v \in V$, $\Gamma \in [0,1]^{E},$ $g_I(\gamma)$ is continuous for all $\gamma \in [0,1]$ and $$\frac{\partial}{\partial \gamma} \frac{g_I(\gamma) - g_I(0)}{\gamma} \geq 0,$$ for all $\gamma \in (0,1)$, where $$g_I(\gamma) \equiv I_2(X^{\Gamma_{e,\gamma}}_u; X^{\Gamma_{e,\gamma}}_v \mid Y^{\Gamma_{e,\gamma}}),$$ and $\Gamma_{e,\gamma}$ denotes the function in $[0,1]^{E}$ such that $\Gamma_{e,\gamma}(e) = \gamma$ and $\Gamma_{e,\gamma}(f) = \Gamma(f)$ for all $f \neq e$.
\end{enumerate}

Then, for any $u,v \in V$, $$I_2(X^{\Gamma}_u; X^{\Gamma}_v\mid Y^{\Gamma}) \leq (|\cG|-1) \cdot \PP(u \sim v \mid \Gamma).$$
\end{theorem}

Another possible extension is concerned with hypergraphs, i.e., interactions of more than two vertex variables per edge.

Finally, while presenting this paper at the Workshop on Combinatorial Statistics, Montreal, May 2018, Y.~Polyanskiy and Y.~Wu informed us of their recent work to appear that obtains results for bounding the classical mutual information for reconstruction problems on graph with bond percolation estimates of properly rescaled open-probability \cite{yury_percolation}. While the results are similar in appearance, parts of the proof techniques appear to be different. In particular, Polyanskiy-Wu  capitalize on their prior general results for strong data-processing inequalities for channels and Bayesian networks \cite{polyanskiy2017strong}. 

\section{Proofs of Theorem \ref{thm:mainsymmetricbound} and Corollary \ref{cor:symmetricbound}}

We first prove a version of Theorem \ref{thm:mainsymmetricbound} for the special case in which all of the edge channels are binary symmetric. We will then extend this specific result to general symmetric channels.

\begin{theorem}\label{thm:mainbinarysymmetricbound}
Let $P_{G,Q}$ be a graphical channel, where $G = (V, E(G))$ is a random graph with vertex set $V = [n]$, and each edge channel $Q_e$ is a binary symmetric channel. Let $(X,Y)$ be a binary synchronization instance drawn from $P_{G,Q}$. Then for all $u,v \in V$, $$I_2(X_u; X_v\mid Y) \leq \PP(u\sim v \mid (I_2(X_i; X_j \mid Y_{(i,j)}))_{(i,j) \in E(G)}).$$

\end{theorem}

\begin{proof}
Suppose we can prove the theorem for the case in which the graph is deterministic. Then, writing $\gamma_{(i,j)} = I_2(X_i; X_j \mid Y_{(i,j)})$, since the graph $G$ is a deterministic function of the edge observations $Y$: $$I_2(X_u; X_v \mid Y) = I_2(X_u; X_v \mid Y, G) \leq \EE_{G}[\PP(u \sim v \mid (\gamma_e)_{e \in E(G)})] = \PP(u \sim v \mid (\gamma_e)_{e \in E(G)}),$$ as desired. Therefore, we may assume that $G$ is deterministic.

For each edge $e \in E(G)$, let the flip probability of $Q_e$ be $\varepsilon_e$, and define $\delta_e = (1-2\varepsilon_e)$. We can assume that $\delta_e \in [0,1]$, because we lose no information by flipping edge labels deterministically. Moreover, by direct calculation $$\gamma_e = \delta_e^2.$$

The proof goes by induction on $|S_{\delta}|$, where $S_{\delta} \equiv \{e \in E(G) : \delta_e \not\in \{0,1\}\} = \{e \in E(G) : \gamma_e \not\in \{0,1\}\}$.

\paragraph{Base case.} $|S_{\delta}| = 0$, so $\gamma_e \in \{0,1\}$ for all $e \in E(G)$. Hence, for all $u,v \in V$, $\PP(u\sim v \mid \gamma) \in \{0,1\}$.
\begin{enumerate}
    \item If $\PP(u\sim v \mid \gamma) = 0$, then let $C_u \subseteq V$ be the open component containing $u$ in the bond percolation on $G$ with open probability $\gamma$. Notice that $C_u$ is deterministic, $v \not\in C_u$, and for all edges $e$ from $C_u$ to $V \sm C_u$, $\delta_e = 0$. Hence $$0 \leq I_2(X_u; X_v \mid Y) \leq I_2((X_w)_{w \in C_u}; (X_w)_{w \in V \sm C_u} \mid Y) = 0.$$
    
    \item If $\PP(u\sim v\mid \gamma) = 1$, then there is a path $P$ from $u$ to $v$ with vertices $u = w_0,w_1,\ldots,w_l = v$ such that $E(P) \subset \{e : \gamma_e = 1\} = \{e : \delta_e = 1\}.$ So $$\prod_{e \in E(P)} Y_e = \prod_{i=0}^{l-1} (X_{w_i} \cdot X_{w_{i+1}}) = X_{w_0} \cdot X_{w_l} = X_u \cdot X_v,$$ which implies that $$I_2(X_u; X_v \mid Y) = I_2(X_u; X_v \mid Y, X_u \cdot X_v) = I_2(X_u; X_u \mid Y, X_u \cdot X_v) = 1,$$ since $X_u$ is $\Rad(1/2)$ and is independent of $Y$.
\end{enumerate}

\paragraph{Inductive step.}
We assume that the theorem for all binary symmetric channels given by $\delta' : E(G) \to [0,1]$ with $|S_{\delta'}| < |S_{\delta}|$.

Pick an arbitrary edge $f \in E(G)$ such that $\delta_f \not\in \{0,1\}$. We will now interpolate between the case in which $\delta_f = 0$, and the case in which $\delta_f = 1$, and the other edge channels are held fixed. For any $t \in [0,1]$, let $\delta^{(t)} : E(G) \to [0,1]$ be given by $$\delta^{(t)} \equiv \begin{cases} \delta_e, & e \neq f \\ t, & e = f \end{cases},$$ let $\gamma^{(t)} = (\delta^{(t)})^2$, let $Q^{(t)}$ be the edge channel corresponding to a binary symmetric channel of parameter associated to $\delta^{(t)}$, and let $(X^{(t)}, Y^{(t)})$ be a binary synchronization instance drawn from $P_{G,Q^{(t)}}.$ Now write $$g_I(t) \equiv I_2(X^{(t)}_u; X^{(t)}_v \mid Y^{(t)})$$ and $$g_P(t) \equiv \PP(u \sim v \mid \gamma^{(t)}).$$
It suffices to prove that $g_I(t) \leq g_P(t)$ for all $t \in [0,1]$.
\\

To this end, for any $t \in [0,1]$, let $B^{(t)}$ be the bond percolation on $G$ with open probability $\gamma^{(t)}$. Couple $B^{(t)}$ with $B^{(0)}$ and $B^{(1)}$ as follows:
$$B^{(t)} = \begin{cases} B^{(0)}, & \mbox{w.p. } 1-t^2 \\ B^{(1)}, & \mbox{w.p. } t^2 \end{cases}$$
In $B^{(0)}$ the probability that $f$ is open is 0, in $B^{(1)}$ the probability that $f$ is open is 1, and in $B^{(t)}$ the probability that $f$ is open is $t^2$. All other edge openness probabilities are independent of $t$. Thus, \begin{equation}\label{eq:connectionprobabilitygrowthbound}g_P(t) = g_P(0) +  (g_P(1)-g_P(0))\cdot t^2.\end{equation}

\begin{claim}\label{claim:termwisegrowthbound}
For any $t \in [0,1]$,
\begin{equation}\label{eq:ggrowthbound}g_I(t) = g_I(0) + (g_I(1) - g_I(0)) \cdot t^2 \cdot h(t),\end{equation} where $h(t)$ is non-decreasing in $[0,1]$.
\end{claim}

We first show that the theorem follows immediately from this claim. Since\footnote{Assuming $g_I(0) \neq g_I(1)$, otherwise the inductive step is trivial.} $h(1) = 1$, and $h(t)$ is non-decreasing for $t \in [0,1]$, we know that $h(t) \leq 1$. This implies that \begin{equation}\label{eq:ggrowthboundexplicit}g_I(t) \leq g_I(0) + (g_I(1) - g_I(0)) \cdot t^2.\end{equation} Then, parametrizing $t$ by $s \in [0,1]$ as $$t(s) = \sqrt{s},$$ by \eqref{eq:connectionprobabilitygrowthbound}, \begin{equation}\label{eq:connectionprobabilitygrowthboundreparametrized} g_P(t(s)) = g_P(0) + (g_P(1) - g_P(0)) \cdot s,\end{equation} and by \eqref{eq:ggrowthboundexplicit}, \begin{equation}\label{eq:ggrowthboundexplicitreparametrized}g_I(t(s)) \leq g_I(0) + (g_I(1) - g_I(0)) \cdot s.\end{equation}

The right-hand sides of equations \eqref{eq:connectionprobabilitygrowthboundreparametrized} and \eqref{eq:ggrowthboundexplicitreparametrized} are equations for line segments for $s \in [0,1]$, and the endpoints of the line segment in \eqref{eq:ggrowthboundexplicitreparametrized} are under the endpoints of the line segments in \eqref{eq:connectionprobabilitygrowthboundreparametrized}, because inductively on $|S_{\delta^{(t)}}|$, we know that \begin{equation*}\label{eq:noisyinduction}
    g_I(0) \leq g_P(0)
\end{equation*} and
\begin{equation*}\label{eq:noiselessinduction}
    g_I(1) \leq g_P(1).
\end{equation*}
Hence, $$g_I(t) \leq g_P(t)$$ for all $t \in [0,1]$.

We now prove the claim.
\begin{proof}[Proof of Claim \ref{claim:termwisegrowthbound}]
Writing $f = (i,j)$, $W^{(t)} = X^{(t)}_u \cdot X^{(t)}_v$, $Z^{(t)} = X^{(t)}_i \cdot X^{(t)}_j$ and $\bar{Y}^{(t)} = (Y^{(t)}_e)_{e \in (E(G) \sm f)}$,

\begin{align*}g_I(t) = \EE[\EE[W^{(t)} \mid \bar{Y}^{(t)}, Y_f^{(t)}]^2] =
\sum_{\sigma \in \{-1,+1\}^{(E(G) \sm f)}} \PP[\bar{Y}^{(t)} = \sigma] \cdot r_{\sigma}(t),\end{align*}
where we take the sum over $\sigma$ such that $\PP[\bar{Y}^{(t)} = \sigma] \neq 0$, and where for each such $\sigma$
\begin{align*}r_{\sigma}(t) \equiv 
 \sum_{\rho \in \{-1,+1\}} \PP[Y_f^{(t)} = \rho \mid \bar{Y}^{(t)} = \sigma ] \left(\sum_{w \in \{-1,+1\}} w  \cdot\PP[W^{(t)} = w \mid \bar{Y}^{(t)} = \sigma, Y_f^{(t)} = \rho]\right)^2.
\end{align*}

Fix $\sigma$ and let \begin{align*}a &= \PP[W^{(t)} = +1, Z^{(t)} = +1 \mid \bar{Y}^{(t)} = \sigma],\\ b &= \PP[W^{(t)} = +1, Z^{(t)} = -1 \mid \bar{Y}^{(t)} = \sigma], \\ c &= \PP[W^{(t)} = -1, Z^{(t)} = +1 \mid \bar{Y}^{(t)} = \sigma], \\ d &= \PP[W^{(t)} = -1, Z^{(t)} = -1 \mid \bar{Y}^{(t)} = \sigma].\end{align*}

We know that \begin{align*}
\PP[W^{(t)} = w, Y_f^{(t)} = \rho \mid \bar{Y}^{(t)} = \sigma] &= \sum_{z \in \{-1,+1\}} \PP[W^{(t)} = w, Y_f^{(t)} = \rho, Z^{(t)} = z \mid \bar{Y}^{(t)} = \sigma] \\ &= \sum_{z \in \{-1,+1\}} \PP[W^{(t)} = w, Z^{(t)} = z \mid \bar{Y}^{(t)} = \sigma]\cdot \PP[ Y_f^{(t)} = \rho \mid Z^{(t)} = z],
\end{align*}
because $Y_f^{(t)}$ is independent of $(W^{(t)},\bar{Y}^{(t)})$ given $Z^{(t)}$.
And also, \begin{equation*}\PP[Y_f^{(t)} = \rho \mid \bar{Y}^{(t)} = \sigma] = \sum_{w \in \{-1,+1\}} \sum_{z \in \{-1,+1\}} \PP[W^{(t)} = w, Z^{(t)} = z \mid \bar{Y}^{(t)} = \sigma] \cdot \PP[Y_f^{(t)} = \rho \mid Z^{(t)} = z].\end{equation*}
Thus,
\begin{align*}r_{\sigma}(t) = \bigg(&  \frac{((a(1-t) + b(1+t)) - (c(1-t)+d(1+t)))^2}{2((a(1-t) + b(1+t)) + (c(1-t)+d(1+t)))} \\ &+\frac{((a(1+t) + b(1-t)) - (c(1+t)+d(1-t)))^2}{2((a(1+t) + b(1-t)) + (c(1+t)+d(1-t)))}\bigg).\end{align*}
In particular,
\begin{equation*}r_{\sigma}(0) = \frac{((a+b)-(c+d))^2}{(a+b)+(c+d)}.\end{equation*}
Hence
$$r_{\sigma}(t) = t^2 \cdot h_{\sigma}(t) + r_{\sigma}(0),$$ where $$h_{\sigma}(t) \equiv \frac{16(ad - bc)^2}{(a+b+c+d)((a+b+c+d)^2-t^2(a-b+c-d)^2)},$$ which is positive and non-decreasing on $[0,1]$, because the numerator and denominator of $h_{\sigma}(t)$ are nonnegative, the numerator is constant, and the denominator is monotonically non-increasing for $t \in [0,1]$. (Since we assume that $(a+b+c+d) > 0$, the case $(a-b+c-d)^2 = (a+b+c+d)^2$ is the only case in which the denominator can be zero, and in this case we can set $h_{\sigma}(t) = 0$, since $Z^{(t)}$ is fully determined by $\bar{Y}^{(t)} = \sigma$ and knowing $Y_f^{(t)}$ can give no new information.)

Since $\PP[\bar{Y}^{(t)} = \sigma] = \PP[\bar{Y}^{(0)} = \sigma]$,
$$g_I(t) - g_I(0) = t^2\cdot \left(\sum_{\sigma \in \{-1,+1\}^{(E(G) \sm f)}}  \PP[\bar{Y}^{(t)} = \sigma] \cdot h_{\sigma}(t)\right),$$ the claim follows by defining $$h(t) = \left(\sum_{\sigma \in \{-1,+1\}^{(E(G) \sm f)}}  \PP[\bar{Y}^{(t)} = \sigma] \cdot h_{\sigma}(t)\right).$$
\end{proof}
The proof of the claim concludes the proof of Theorem \ref{thm:mainbinarysymmetricbound}.
\end{proof}




In order to see the relationship between Theorem \ref{thm:mainbinarysymmetricbound} and Theorem \ref{thm:mainsymmetricbound}, we define the ``absolute value'' of the output of a symmetric edge channel:

\begin{definition}\label{def:absolutevalue}
Given a symmetric edge channel $Q$ with output alphabet $\cY$ and symmetry transformation $T : \cY \to \cY$, we define the absolute value $|\cdot|_T : \cY \to 2^{\cY}$ by $$|y|_T = \{y, T(y)\}.$$
\end{definition}

The definition is motivated by viewing $T$ as a sign-flipping transformation (which it is, in the $\BSC$ and $\AWGN$ cases). Notice that since $T^2 = \mbox{id}$, $|y|_T = |T(y)|_T$ for all $y \in \cY$.

We are now ready to prove Theorem \ref{thm:mainsymmetricbound}.

\begin{proof}[Proof of Theorem \ref{thm:mainsymmetricbound}]
\item 
\paragraph{Proof overview}
We are given a binary synchronization instance $(X,Y)$ drawn from a symmetric graphical channel $P_{G,Q}$ such that each edge channel $Q_e$ has symmetry transformation $T_e$. In the proof that follows, we will (1) show that if we first reveal $|Y|_T$, then we do not give away information about $X$ -- indeed, the posterior distribution of $X$ given $|Y|_T$ will still be uniform over $\{-1,1\}^V$. We will then (2) show that conditioned on $|Y|_T$, $(X,Y)$ is a \textit{binary} symmetric synchronization instance. The theorem will then follow by (3) applying the bound of Theorem \ref{thm:mainbinarysymmetricbound} for each realization $|Y|_T = z$, and by taking the expectation over $|Y|_T$ at the very end.

\paragraph{Notation and Conditional Probability Calculations}
With a slight abuse of notation, for every edge $e$ let $$Q_{e}(\cdot) \equiv \frac{Q_{e \mid +}(\cdot) + Q_{e \mid -}(\cdot)}{2},$$ and define $g_e^+ \equiv \frac{1}{2}\frac{dQ_{e\mid +}}{dQ_{e}}$ and $g_e^- \equiv \frac{1}{2}\frac{dQ_{e\mid -}}{dQ_{e}}$, so that $g_e^+ + g_e^- \equiv 1$.
\\

The motivation for these definitions is that for all $y_e \in \cY_e$,

\begin{equation}\label{eq:XgivenYpos}\PP[X_e = +1 \mid Y_e = y_e] = \left(\frac{d(Q_{e \mid +}/2)}{d((Q_{e \mid +} + Q_{e \mid -})/2)}\right) \left(y_e\right) = \left(\frac{dQ_{e \mid +}}{d(Q_{e \mid +} + Q_{e \mid -})}\right) \left(y_e\right) =  g_e^+(y_e).\end{equation} Similarly, \begin{equation}\label{eq:XgivenYneg}\PP[X_e = -1 \mid Y_e = y_e] = g_e^-(y_e).\end{equation}

Notice also that since the channel $Q_e$ is symmetric with transformation $T_e$,
$$Q_e'(\cdot) \equiv Q_e(T_e(\cdot)) = \frac{Q_{e \mid +}(T_e(\cdot)) + Q_{e \mid -}(T_e(\cdot))}{2} = \frac{Q_{e \mid -}(\cdot) + Q_{e \mid +}(\cdot)}{2} = Q_e(\cdot).$$ This means that
\begin{equation}\label{eq:YgivenXsymmetry}g_e^+(y_e) = g_e^-(T_e(y_e)).\end{equation}
Also, if $y_e \neq T_e(y_e)$,
\begin{align}\label{eq:YgivenZ}\PP[Y_e = y_e \mid |Y_e|_{T_e} = |y_e|_{T_e}] 
&= \left(\frac{dQ_e}{d(Q_e + Q'_e)}\right)(y_e) = \left(\frac{1}{2}\frac{dQ_e}{dQ_e}\right)(y_e) = \frac{1}{2}. \end{align}
and if $y_e = T_e(y_e)$, then
\begin{equation}\label{eq:YgivenZeqtriv}
\PP[Y_e = y_e \mid |Y_e|_{T_e} = |y_e|_{T_e}] = 1.
\end{equation}

\paragraph{(1) $|Y|_T$ is independent of $X$}
For each $e \in E(G)$, $|Y_e|_{T_e}$ is the output of a channel on $X_e$, so it suffices to prove that for all $e$, $$\PP[X_e = +1 \mid |Y_e|_{T_e} = z_e] = \frac{1}{2}.$$

Writing $z_e = \{y_e, T_e(y_e)\}$, if $y_e = T_e(y_e)$, then $$\PP[X_e = +1 \mid |Y_e|_{T_e} = z_e] = \PP[X_e = +1 \mid Y_e = y_e] = g_e^+(y_e)
= \frac{g_e^+(y_e) + g_e^-(T_e(y_e))}{2} = \frac{g_e^+(y_e) + g_e^-(y_e)}{2} = \frac{1}{2}.$$
And if $y_e \neq T_e(y_e)$, then 
\begin{align}\PP[X_e = +1 \mid |Y_e|_{T_e} = z_e] &= \sum_{y'_e \in z_e} \PP[X_e = +1 \mid Y_e = y'_e] \PP[Y_e = y'_e \mid |Y_e|_{T_e} = z_e] \\ \label{eq:criticalstepinZconditioningXindependence} &= g_e^+(y_e) \cdot \frac{1}{2} + g_e^+(T_e(y_e)) \cdot \frac{1}{2} \\ &= \label{eq:conditioningsymmetrystepXgivenZ} \frac{1}{2}\left(g_e^+(y_e) + g_e^-(y_e)\right) \\ \label{eq:XgivenZ} &= \frac{1}{2}.\end{align}
Line \eqref{eq:criticalstepinZconditioningXindependence} follows by plugging in \eqref{eq:XgivenYpos} and  \eqref{eq:YgivenZ}. Line \eqref{eq:conditioningsymmetrystepXgivenZ} follows by \eqref{eq:YgivenXsymmetry}.

\paragraph{(2) Conditioned on $|Y|_T = z$,  $(X,Y)$ is distributed as a binary synchronization instance} drawn from $P_{G,Q^z}$, where $Q^z$ is the sequence of unary and binary symmetric edge channels $Q_e^{z}$ with output alphabet $z_e$, given by

$$Q_e^{z}(y_e \mid x_e) = \delta(\{y_e\} = z_e)$$ if $|z_e| = 1$, and
$$Q_e^{z}(y_e \mid x_e) = \begin{cases}0, & y_e \not\in z_e \\ g_e^+(y_e), & y_e \in z_e, x_e = +1 \\ g_e^-(y_e), & y_e \in z_e, x_e = -1\end{cases}$$ if $|z_e| = 2$.

Notice that if $|z_e| = 1$, then $Q_e^{z}$ is a unary-output channel. And if $|z_e| = 2$, then $Q_e^{z}$ is a binary symmetric channel, because \begin{enumerate}
    \item \textbf{Binary}: The output alphabet is $z_e = \{y_e, T_e(y_e)\}$, and $y_e \neq T_e(y_e)$.
    \item \textbf{Well-defined}: $g_e^+(y_e) + g_e^+(T_e(y_e)) = g_e^+(y_e) + g_e^-(y_e) = 1$, and $g_e^-(y_e) + g_e^-(T_e(y_e)) = g_e^+(T_e(y_e)) + g_e^-(T_e(y_e)) = 1$.
    \item \textbf{Symmetric}: $g_e^+(y_e) = g_e^-(T_e(y_e))$ and $g_e^+(T_e(y_e)) = g_e^-(y_e)$.
\end{enumerate}

Since we have shown that $X$ is independent of $|Y|_T$ in the previous paragraph, we only need to that conditioned on $|Y|_T = z$ and  $X = x$, the random variable $Y$ has distribution $P_{G,Q^z}(\cdot \mid x)$: We write for all $x,y,z$,
\begin{align*}
    \PP[Y = y \mid X = x, |Y|_T = z] = \prod_{e} \PP[Y_e = y_e \mid X_e = x_e, |Y_e|_{T_e} = z_e],
\end{align*} since each edge label $Y_e$ is independent of the other edge labels given $(X,|Y|_T)$, and each edge label $Y_e$ depends only on $(X_e,|Y_e|_{T_e})$ given $(X,Y)$.

So we just need to prove that for each edge $e$, $$\PP[Y_e = y_e \mid X_e = x_e, |Y_e|_{T_e} = z_e] = Q_e^{z}(y_e \mid x_e).$$ This is true, because if $T_e(y_e) = y_e$, then $$\PP[Y_e = y_e \mid X_e = x_e, |Y_e|_{T_e} = z_e] = \delta(y_e \in z_e) = Q_e^{z}(y_e \mid x_e),$$ and if $T_e(y_e) \neq y_e$, then \begin{align}\PP[Y_e = y_e \mid X_e = x_e, |Y_e|_{T_e} = z_e] &= \frac{\PP[X_e = x_e \mid Y_e = y_e] \cdot \PP[Y_e = y_e \mid |Y_e|_{T_e} = z_e]}{\PP[X_e = x_e \mid |Y_e| = y_e]} \\ &= \frac{\PP[X_e = x_e \mid Y_e = y_e] \cdot \PP[Y_e = y_e \mid |Y_e|_{T_e} = z_e]}{\PP[X_e = x_e]} \\ \label{eq:applyYgivenZ} &= \frac{\PP[X_e = x_e \mid Y_e = y_e]}{2\PP[X_e = x_e]} \\ \label{eq:applyXgivenY} &= Q_e^{z_e}(y_e \mid x_e). \end{align}


\paragraph{(3) Application of Theorem \ref{thm:mainbinarysymmetricbound}}

For each $e = (i,j) \in E(G)$, define $$\gamma^z_e = I_2(X_i; X_j \mid Y_e, |Y_e|_{T_e} = z_e),$$ and $$\gamma_e = I_2(X_i; X_j \mid Y_e).$$ Then
\begin{align} \label{eq:mainsymmetricconditiononz} I_2(X_u; X_v \mid Y) &= I_2(X_u; X_v \mid Y, |Y|_T) \\ &= \EE_{|Y|_T}[I_2(X_u; X_v \mid Y)] \\ \label{ineq:applybsccase} &\leq \EE_{|Y|_T}[\PP(u \sim v \mid \gamma^{|Y|_T})] \\ &=\label{eq:expectationofpercolation} \PP(u \sim v \mid \gamma).
\end{align}

Line \eqref{eq:mainsymmetricconditiononz} follows because $|Y|_T$ is a function of $Y$.

Line \eqref{ineq:applybsccase} follows by Theorem \ref{thm:mainbinarysymmetricbound}, because conditioned on $|Y|_T = z$, the pair $(X,Y)$ is a binary synchronization instance drawn from a graphical channel with binary symmetric edge channels. (We can ignore the unary edge channels, since their outputs do not depend on $X$, and if $Q^{z}_e$ is unary, then $\gamma^z_e = 0$, so $e$ does not influence the probability of a connection $u \sim v$ in the percolation.)

Line \eqref{eq:expectationofpercolation} follows because we can let $P$ be a bond percolation on $G$ such that, each edge $e$ is independently open with probability $\gamma^{|Y|_T}_e$. Then the probability that $u$ and $v$ are connected by an open path is $\EE_{|Y|_T}[\PP(u\sim v \mid \gamma^{|Y|_T})]$. We can also calculate this probability in a different way: each edge $e = (i,j)$ in $P$ is {\em independently} open with probability $\EE_{|Y|_T}[\gamma^{|Y|_T}_e] = I_2(X_i; X_j \mid Y_e,|Y_e|_{T_e}) = \gamma_e,$ where the independence occurs because the entries of $|Y|_T$ are all independent, since they are independent of each other given $X$, and $|Y|_T$ is independent of $X$. Hence, there is an open path in $P$ connecting $u$ and $v$ with probability $\PP(u\sim v \mid \gamma)$.
\end{proof}

We can now extend Theorem \ref{thm:mainsymmetricbound} to bound the information that the edge labels $Y$ and a set $X_S$ of vertex label give about another vertex label $X_u$:
\begin{proof}[Proof of Corollary \ref{cor:symmetricbound}]
Create a ``virtual'' vertex $w$ and construct the graph $G'$ with $V(G') = V(G) \cup w$ and $E(G') = E(G) \cup \{(v,w) : v \in S\}$. Let $Q'$ be edge channels such that $Q'_e = Q_e$ for all $e \in E(G)$, and $Q'_{(v,w)}(y \mid x) = \delta(x = y)$ for all $v \in S$. Draw $(X',Y')$ from $P_{G',Q'}$. Since $P_{G',Q'}$ is a symmetric graphical channel, by Theorem \ref{thm:mainsymmetricbound}, $$I_2(X'_u; X'_w \mid Y') \leq \PP(u \sim w \mid (I_2(X'_i; X'_j \mid Y'_{(i,j)}))_{(i,j) \in E(G')}) = \PP(u \sim S \mid (I_2(X_i; X_j \mid Y_{(i,j)}))_{(i,j) \in E(G)}).$$
And $$I_2(X'_u; X'_w \mid Y') = I_2(X'_u; X'_S \mid Y') = I_2(X_u; X_S\mid Y).$$
\end{proof}


\bibliographystyle{alpha}
\bibliography{gen_sbm.bib}

\appendix

\section{$\chi^2$-mutual information}\label{app:chi2mutualinformation}
In this appendix, we define the $\chi^2$-mutual information, $I_2$, and prove Propositions \ref{prop:chi2informationequivalences}

\subsection{$f$-divergences and $f$-mutual informations}
\paragraph{$f$-divergences.}
Given two probability distributions $\mu$ and $\nu$ over a probability space $\Omega$ such that $\mu \ll \nu$ (that is, $\mu$ is absolutely continuous with respect to $\nu$), and given strictly convex $f : \RR \to \RR$ such that $f(1) = 0$, we may define the $f$-divergence $$D_f(\mu || \nu) \equiv \int_{\Omega} f\left(\frac{d\mu}{d\nu}\right) d\nu.$$ Here $\frac{d\mu}{d\nu}$ denotes the Radon-Nikodym derivative. (When $\Omega$ is finite, $\frac{d\mu}{d\nu} (x) = \frac{\mu(x)}{\nu(x)}$ for all $x \in \Omega$.) $f$-divergences were introduced in \cite{csiszar1967information}.
\paragraph{$f$-mutual informations.}
Given variables $A,B$ with joint distribution $\nu_{A,B}$ on $\cA \times \cB$, and marginal distributions $\nu_A$ on $\cA$, $\nu_B$ on $\cB$, the $f$-mutual information between them is given by 
$$I_f(A; B) \equiv D_f(\nu_{A,B} || (\nu_{A} \times \nu_{B})).$$
$I_f$ is nonnegative, and zero if and only if $A$ and $B$ are independent. Thus, we can take it as a measure of the degree of independence of the variables $A$ and $B$: the higher the mutual information, the more ``correlated'' the variables are, and the more information they give about each other.

Moreover, the $f$-mutual information also has following well-known ``data-processing'' property (see \cite{CoverThomas}, for example):
\begin{proposition}\label{prop:dataprocessing}
for $A,B,C$ such that $A$ is independent of $C$ given $B$, \begin{equation}I_f(A; C) \leq I_f(A; B).\end{equation}
In particular, if $C$ is a deterministic function of $B$, then $I_f(A; C) \leq I_f(A; B).$
\end{proposition}

\subsection{Definition and basic properties of the $\chi^2$-mutual information}
\begin{definition}The $\chi^2$-mutual information, $I_2$, is the $f$-mutual information, $I_f$, with $f(t) = (t - 1)^2$.
\end{definition}

\begin{proposition}\label{prop:correqualschi2}
Let $A,U$ be jointly-distributed random variables, where $U \sim \Rad(1/2)$. Then $$I_2(A; U) = \EE_A[\EE_U[U \mid A]^2].$$
\end{proposition}
\begin{proof}
For $U \sim \Rad(1/2)$, $$I_2(A; U) = \EE_A[\EE_U[U \mid A]^2]$$ because, letting $\nu_Z$ denote the distribution of $Z$, and $\Omega$ denote the sample set of $A$,
\begin{align*}
    I_2(A; U) &= \int_{\Omega \times \{-1,+1\}} \left(\frac{d(\nu_{A,U})}{d(\nu_A \times \nu_U)} - 1\right)^2 d(\nu_{A} \times \nu_{U})
    \\ &= \int_{\Omega \times \{-1,+1\}} \left(\frac{d(\nu_{A,U} - \nu_{A} \times \nu_U)}{d(\nu_A \times \nu_U)}\right)^2 d(\nu_A \times \nu_U) \\ &= \int_{\Omega \times \{-1,+1\}} \left(\frac{d(\nu_{A,U} - (\nu_{A,U} + \nu_{A,-U})/2)}{d(\nu_A \times \nu_U)}\right)^2 d(\nu_A \times \nu_U) \\ &= \int_{\Omega \times \{-1,+1\}} \left(\frac{1}{2}\frac{d(\nu_{A,U} - \nu_{A,-U})}{d(\nu_A \times \nu_U)}\right)^2 d(\nu_A \times \nu_U) \\ &= \int_{\Omega} \EE_U[U \mid A]^2 d\nu_A \\ & = 
    \EE_A[\EE_U[U \mid A]^2].
\end{align*}
\end{proof}

\subsection{Proof of Proposition \ref{prop:chi2informationequivalences}}
\begin{proof}[Proof of Proposition \ref{prop:chi2informationequivalences}]
For a binary synchronization instance $(X,Y)$ with underlying graph $G$, and for $u \in V(G)$, $S \subseteq V(G)$, we have $$I_2(X_u; X_S, Y) = I_2(X_u; X_S \mid Y)$$ because \begin{align}\label{eq:applycorrequalschi2firsttimefirstI2equivalence}
    I_2(X_u; X_S, Y) &= \EE_{X_S,Y}[\EE_{X_u}[X_u \mid X_S, Y]^2] \\ &= \EE_{Y}[\EE_{X_S}[\EE_{X_u}[X_u \mid X_S, Y]^2 \mid Y]] \\
    &= \label{eq:applycorrequalschi2secondtimefirstI2equivalence} I_2(X_u; X_S \mid Y),
\end{align}
where Lines \eqref{eq:applycorrequalschi2firsttimefirstI2equivalence} and \eqref{eq:applycorrequalschi2secondtimefirstI2equivalence} follow by Proposition \ref{prop:correqualschi2}, as $X_u$ is $\Rad(1/2)$ and is independent of $Y$.
\end{proof}

\subsection{Proof of Proposition \ref{prop:chi2binarysynchexpectationexpression}}
\begin{proof}[Proof of \ref{prop:chi2binarysynchexpectationexpression}]
By Proposition \ref{prop:chi2informationequivalences} it is equivalent to show that $$I_2(X_u; X_v, Y) = \EE_Y[\EE_X[X_u \cdot X_v \mid Y]^2].$$ By Proposition \ref{prop:correqualschi2} it suffices to show that $$I_2(X_u; X_v, Y) = I_2(X_u \cdot X_v; Y).$$ This is true because $$I_2(X_u; X_v,Y) = I_2(X_u \cdot X_v; X_v,Y) = I_2(X_u \cdot X_v; Y).$$ The first equality holds by data-processing (Proposition \ref{prop:dataprocessing}). The second holds because $X_v$ is independent of $(X_u\cdot X_v,Y)$.
\end{proof}

\subsection{$\chi^2$-mutual information and KL-mutual information}
\begin{definition}
The KL-mutual information, $I_{KL}$, is the $f$-mutual information, $I_f$, with $f(t) = t \lg t$. $\lg t \equiv (\log t)  / (\log 2)$ is the base-2 logarithm.
\end{definition}
\begin{lemma}[$I_2$ vs. $I_{\mathrm{KL}}$]\label{lem:evanschi2vsklonradhalf}
Let $A,U$ be jointly-distributed random variables, where $U \sim \Rad(1/2)$. Then $$\frac{1}{2}I_2(A; U) \leq I_{\mathrm{KL}}(A; U) \leq I_2(A; U).$$
\end{lemma}
This follows from the following inequalities 
$$\frac{x^2}{2} \leq \frac{1+x}{2} \lg(1-x) + \frac{1-x}{2} \lg(1+x) \leq x^2.$$

\end{document}